\documentclass[preprint,12pt,authoryear]{elsarticle}

\usepackage{amssymb,amsthm,amsmath,amsfonts,mathrsfs}%General
\usepackage{bm}
\usepackage{cancel}
\usepackage{enumerate}

\usepackage{multirow}%Table
\usepackage{boldline}
\usepackage{array}

\usepackage{chessboard}%Hackenbush
\usepackage{pst-all}

\usetikzlibrary{decorations.pathmorphing}%Tree
\RequirePackage{etoolbox}

\newcommand{\GL}{{G^\mathcal{L}}}%Games
\newcommand{\GR}{{G^\mathcal{R}}}

\newcommand{\cgstar}{\mathord\ast}
\newcommand{\cgup}{\mathord\uparrow}
\newcommand{\cgupstar}{\cgup\cgstar}
\newcommand{\cgdoubleup}{\mathord\Uparrow}
\newcommand{\cgallstar}{\mathord\circledast}
\newcommand{\os}{\mathbin:}
\newcommand{\cggfuz}{\mathrel{\rule[-.05ex]{.05ex}{1.3ex}\hspace{.8pt}\rhd}}
\newcommand{\cglfuz}{\mathrel{\lhd\hspace{.8pt}\rule[-.05ex]{.05ex}{1.3ex}}}
\newcommand{\cgfuzzy}{\mathrel\|}
\newcommand{\LS}[1]{\text{LS}({#1})}
\newcommand{\LSI}[1]{\text{\emph{LS}}({#1})}
\newcommand{\RS}[1]{\text{RS}({#1})}
\newcommand{\RSI}[1]{\text{\emph{RS}}({#1})}
\newcommand{\LD}[1]{\text{Ld}({#1})}
\newcommand{\LDI}[1]{\text{\emph{Ld}}({#1})}
\newcommand{\RD}[1]{\text{Rd}({#1})}
\newcommand{\RDI}[1]{\text{\emph{Rd}}({#1})}

\newtheorem{theorem}{Theorem}%Environments

\newtheorem{corollary}[theorem]{Corollary}
\newtheorem{lemma}[theorem]{Lemma}
\newtheorem{definition}[theorem]{Definition}

\newtheorem{observation}[theorem]{Observation}

\emergencystretch=\maxdimen
\journal{Journal of Combinatorial Theory, Series B}

\begin{document}

\begin{frontmatter}

\title{Indecomposable combinatorial games}

\author{Michael Fisher}
\ead{mfisher@wcupa.edu}
\affiliation{organization={West Chester University},
            country={United States of America}}

\author{Neil A. McKay}
\ead{neil.mckay@unb.ca}
\affiliation{organization={University of New Brunswick, Saint John},
            country={Canada}}

\author{Rebecca Milley\fnref{fn1}}
\fntext[fn1]{Funded by Natural Sciences and Engineering Research Council of Canada.}
\ead{rmilley@grenfell.mun.ca}
\affiliation{organization={Grenfell Campus, Memorial University},
            country={Canada}}

\author{Richard J. Nowakowski}
\ead{r.nowakowski@dal.ca}
\affiliation{organization={Dalhousie University},
            country={Canada}}

\author{Carlos P. Santos\corref{correspondingauthor}}
\cortext[correspondingauthor]{Corresponding author. Carlos Santos' work is funded by national funds through the FCT - Funda\c{c}\~{a}o para a Ci\^{e}ncia e a Tecnologia, I.P., under the scope of the projects UIDB/00297/2020 and UIDP/00297/2020 (Center for Mathematics and \linebreak Applications).}
\ead{cmf.santos@fct.unl.pt}
\affiliation{organization={Center for Mathematics and Applications (NovaMath), FCT NOVA},
            country={Portugal}}

\begin{abstract}
In Combinatorial Game Theory, short game forms are defined recursively over all the positions the two players are allowed to move to. A form is decomposable if it can be expressed as a disjunctive sum of two forms with smaller birthday. If there are no such summands, then the form is indecomposable. The main contribution of this document is the characterization of the indecomposable nimbers and the characterization of the indecomposable numbers. More precisely, a nimber is indecomposable if and only if its size is a power of two, and a number is indecomposable if and only if its absolute value is less or equal than one.
\end{abstract}


\end{frontmatter}

\section{Introduction}
\label{sec:intro}
We assume that the reader is acquainted with the basic concepts of short, two-person perfect information combinatorial games (CGT) as presented in any of \cite{ANW019,BCG001,Con001,Sie013}. We only consider normal play where the player who cannot move loses. Indeed, little more than the theory of \textsc{nim} \citep{Bou02} -- nimbers -- and the theory of \textsc{blue-red hackenbush} \citep{BCG001} -- numbers -- is required.

 A central theme of CGT is to simplify the analysis of positions. One way is to replace a position by the smallest (in a game-tree sense) equivalent position, known as the canonical form\footnote{To avoid confusion because of the many informal meanings of `game', we use the intuitively obvious terms `ruleset' and position. We use `game form', or just `game' or `form', for the mathematical object describing a position.}. A second more important way arises when positions decompose into two or more independent components, written $G_1+G_2+\cdots+G_k$, and a player is only allowed to play in one component. In that situation, the CGT  approach is to analyze the canonical form of each component individually, and then give rules, or very good heuristics,  for choosing the best in which to play.

For example, the classic ruleset \textsc{nim} was originally defined as a sum of components. In other rulesets like \textsc{go, domineering,} or \textsc{konane} \citep{Sie013}, parts of the board often become isolated from each other. Our question is the following: \textit{when can we take a canonical form and, usefully, express it as a sum?} The components of the sum should be simpler than the original, so that the analysis becomes easier. That means that the birthdays of the components should be smaller than the birthday of the sum -- game tree height, denoted by $b(G)$. By defining the concept of decomposability as follows, we will be able to show, by the end of this paper, that analyzing endgames with only numbers as components or analyzing \textsc{nim} can be viewed as the breaking of all the components into ``indecomposable components''.\\

\vspace{-0.2cm}
\begin{definition} \label{def:dec}
A game form $G$ is \emph{decomposable} if there are $H$ and $J$ such that $b(H)<b(G)$, $b(J)<b(G)$ and $G=H+J$. If there are no such summands, then the game form $G$ is \emph{indecomposable}.
\end{definition}

The main contribution of this paper is the characterization of the indecomposable nimbers and the characterization of the indecomposable numbers. More precisely, a number is indecomposable if and only if its absolute value is less or equal than one -- Theorem \ref{th:numbers} -- and a nimber is indecomposable if and only if its size is a power of two -- Theorem \ref{th:nimbers}.

Although in general the decompositions are not unique, it is possible to define the concept of ``strong decomposition'', which, in a way, is the simplest of all.

\begin{definition}
A game $G$ is \emph{strongly decomposable} if there are $H$ and $J$ different than zero such that $b(H)+b(J)=b(G)$ and $G=H+J$.
\end{definition}

Observe that a game may be decomposable without being strongly decomposable. For example, if we take $G=\{0\,|\,\cgstar=2\}$, $H=*2$ and $J=\cgupstar$, then we find that $b(G)=3$, $b(H)=2$, $b(J)=2$, $G=H+J$, and $G$ is decomposable.
However, an exhaustive search allows us to conclude that $G$ is not strongly decomposable. We will also characterize the strong decompositions of numbers and nimbers -- see Theorems \ref{th:strongnumbers} and \ref{th:strongnimbers}.\\

Knowing that a form is indecomposable can give information about the outcome of a disjunctive sum. Consider the \textsc{blue-red-green hackenbush} position shown at Figure~\ref{fig:fig1}. The component $G$ is $*8$. On the other hand, the birthday of $H$ is less than or equal to $6$ and that  of $J$ is less than or equal to $7$ (number of edges of each component). Since $G$ is indecomposable ($8$ is a power of $2$), we cannot have $G=H+J$ therefore $G+H+J\ne 0$, and consequently,  $G+H+J$ is not a $\mathcal{P}$-position. That is, at least one of the players must have a winning first move. Here, we concluded that there is a winning move for one of the players using only an algebraic result, without considering game strategies. This type of argument can be used in all rulesets whose birthdays of the positions can be  naturally bounded (placement games, games whose moves are piece-captures, etc.).

 \begin{figure}[htbp]
\label{fig:fig1}
\begin{center}
\scalebox{0.7}{
\definecolor{ffqqqq}{rgb}{1.,0.,0.}
\definecolor{qqqqff}{rgb}{0.,0.,1.}
\definecolor{qqffqq}{rgb}{0.,1.,0.}
\definecolor{xdxdff}{rgb}{0.49019607843137253,0.49019607843137253,1.}
\begin{tikzpicture}
\clip(3.4,-6.6) rectangle (7.2,-0.3);
\draw [line width=2.8pt,color=qqffqq] (4.,-6.)-- (4.,-5.32);
\draw [line width=2.8pt,color=qqffqq] (4.,-5.32)-- (4.,-4.64);
\draw [line width=2.8pt,color=qqffqq] (4.,-4.64)-- (4.,-3.96);
\draw [line width=2.8pt,color=qqffqq] (4.,-3.96)-- (4.,-3.28);
\draw [line width=2.8pt,color=qqffqq] (4.,-3.28)-- (4.,-2.6);
\draw [line width=2.8pt,color=qqffqq] (4.,-2.6)-- (4.,-1.92);
\draw [line width=2.8pt,color=qqffqq] (4.,-1.92)-- (4.,-1.24);
\draw [line width=2.8pt,color=qqffqq] (4.,-1.24)-- (4.,-0.56);
\draw [line width=2.8pt,color=qqqqff] (6.,-1.24)-- (6.,-1.92);
\draw [line width=2.8pt,color=ffqqqq] (6.,-1.92)-- (6.,-2.6);
\draw [line width=2.8pt,color=qqqqff] (6.,-2.6)-- (6.,-3.28);
\draw [line width=2.8pt,color=qqffqq] (6.,-3.28)-- (6.,-3.96);
\draw [line width=2.8pt,color=qqffqq] (6.,-3.96)-- (6.,-4.64);
\draw [line width=2.8pt,color=qqqqff] (6.,-4.64)-- (6.,-5.32);
\draw [line width=2.8pt,color=ffqqqq] (6.,-5.32)-- (6.,-6.);
\draw [line width=2.8pt,color=qqqqff] (5.,-1.92)-- (5.,-2.6);
\draw [line width=2.8pt,color=qqffqq] (5.,-2.6)-- (5.,-3.28);
\draw [line width=2.8pt,color=qqqqff] (5.,-3.28)-- (5.,-3.96);
\draw [line width=2.8pt,color=qqffqq] (5.,-3.96)-- (5.,-4.64);
\draw [line width=2.8pt,color=ffqqqq] (5.,-4.64)-- (5.,-5.32);
\draw [line width=2.8pt,color=qqqqff] (5.,-5.32)-- (5.,-6.);
\begin{scriptsize}
\draw [fill=white] (4.,-5.32) circle (2.5pt);
\draw [fill=white] (5.,-5.32) circle (2.5pt);
\draw [fill=white] (6.,-5.32) circle (2.5pt);
\draw [fill=white] (4.,-4.64) circle (2.5pt);
\draw [fill=white] (5.,-4.64) circle (2.5pt);
\draw [fill=white] (6.,-4.64) circle (2.5pt);
\draw [fill=white] (4.,-3.96) circle (2.5pt);
\draw [fill=white] (5.,-3.96) circle (2.5pt);
\draw [fill=white] (6.,-3.96) circle (2.5pt);
\draw [fill=white] (4.,-3.28) circle (2.5pt);
\draw [fill=white] (5.,-3.28) circle (2.5pt);
\draw [fill=white] (6.,-3.28) circle (2.5pt);
\draw [fill=white] (4.,-2.6) circle (2.5pt);
\draw [fill=white] (5.,-2.6) circle (2.5pt);
\draw [fill=white] (6.,-2.6) circle (2.5pt);
\draw [fill=white] (4.,-1.92) circle (2.5pt);
\draw [fill=white] (5.,-1.92) circle (2.5pt);
\draw [fill=white] (6.,-1.92) circle (2.5pt);
\draw [fill=white] (4.,-1.24) circle (2.5pt);
\draw [fill=white] (6.,-1.24) circle (2.5pt);
\draw [fill=white] (4.,-0.56) circle (2.5pt);
\end{scriptsize}
\draw [line width=3.6pt] (3.,-6.)-- (7.,-6.);
\draw (3.72,-6.1) node[anchor=north west] {$G$};
\draw (3.72+1,-6.1) node[anchor=north west] {$H$};
\draw (3.72+1+1.05,-6.1) node[anchor=north west] {$J$};
\end{tikzpicture}
}
\caption{Since $G=*8$ is indecomposable, we know that $G\neq H+J$ and $G+H+J\not\in\mathcal{P}$.}
\end{center}
\end{figure}
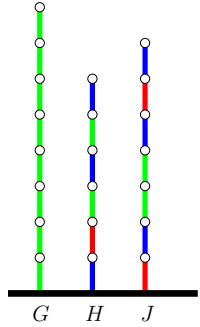

  \subsection{Background}\label{subsec:background}

It is well known that if $G^{\mathcal{L}}$ and $G^{\mathcal{R}}$ are all numbers such that all elements of the first are strictly less than all elements of the second, then the game form $\{\GL \! \mid \! \GR \}$ is the simplest number between the maximum element of $G^{\mathcal{L}}$ and the minimum element of $G^{\mathcal{R}}$ \citep{ANW019,BCG001,Con001,Sie013}. This construction is very reminiscent of Dedekind's construction, but with a recursive nature; in both cases, new numbers are formed in the gaps between ``cuts'' of simpler ones. When the construction is naturally extended beyond short games, we get the surreal numbers \citep{Knu74}, which include the reals, the ordinals, and much more (Figure \ref{fig:fig2}). Considering only short games, as done in this paper, we have only dyadics, i.e., irreducible fractions whose denominators are powers of $2$.

 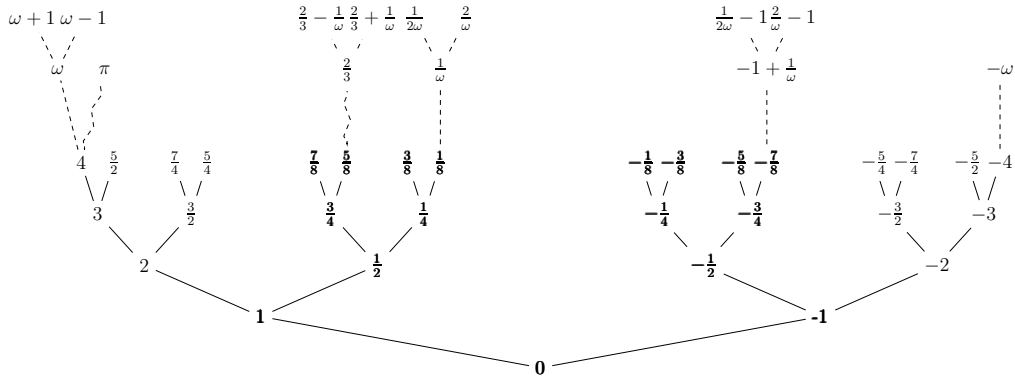
\begin{figure}[htbp]
\begin{center}
\scalebox{0.62}{
\begin{tikzpicture}[level distance=2cm,
level 1/.style={sibling distance=12cm, level distance=1.1cm},
level 2/.style={sibling distance=5cm, level distance=1.1cm},
level 3/.style={sibling distance=2cm, level distance=1.1cm},
level 4/.style={sibling distance=0.7cm, level distance=1.1cm},
level 5/.style={sibling distance=1cm, level distance=2cm},
level 6/.style={sibling distance=1.1cm, level distance=1.1cm}]
\tikzset{dashed edge/.style={dash pattern=on 2pt off 2pt, draw=black, line cap=round, shorten <=2pt, shorten >=2pt}}
\tikzset{
  myzigzagedge/.style={
    decorate,
    decoration={
      zigzag,
      segment length=20,
      amplitude=2,
      post=lineto,
      post length=2pt,
    },
    draw=black,
    shorten <=2pt,
    shorten >=2pt,
  }
}

\node {\pmb{0}}[grow=up]
         child {node {\pmb{-1}}
                      child{node {$-2$}
                                 child{node{$-3$}
                                             child{node {$-4$}
                                                        child[dashed]{node {$-\omega$}
                                                                     }
                                                  }
                                             child{node {$-\frac{5}{2}$}
                                                  }
                                      }
                                 child{node {$-\frac{3}{2}$}
                                             child{node {$-\frac{7}{4}$}
                                                  }
                                             child{node {$-\frac{5}{4}$}
                                                  }
                                      }
                           }
                      child{node {$\pmb{-\frac{1}{2}}$}
                                 child{node {$\pmb{-\frac{3}{4}}$}
                                             child{node {$\pmb{-\frac{7}{8}}$}
                                                        child[dashed]{node {$-1+\frac{1}{\omega}$}
                                                                            child[dashed]{node {$\frac{2}{\omega}-1$}
                                                                                         }
                                                                            child[dashed]{node {$\frac{1}{2\omega}-1$}
                                                                                         }
                                                                     }
                                                  }
                                             child{node {$\pmb{-\frac{5}{8}}$}
                                                  }
                                      }
                                 child{node {$\pmb{-\frac{1}{4}}$}
                                             child{node {$\pmb{-\frac{3}{8}}$}
                                                  }
                                             child{node {$\pmb{-\frac{1}{8}}$}
                                                  }
                                      }
                           }
               }
         child {node {\pmb{1}}
                      child{node {$\pmb{\frac{1}{2}}$}
                                 child{node {$\pmb{\frac{1}{4}}$}
                                             child{node {$\pmb{\frac{1}{8}}$}
                                                        child[dashed]{node {$\frac{1}{\omega}$}
                                                                            child[dashed]{node {$\frac{2}{\omega}$}
                                                                                         }
                                                                            child[dashed]{node {$\frac{1}{2\omega}$}
                                                                                         }
                                                                     }
                                                  }
                                             child{node {$\pmb{\frac{3}{8}}$}
                                                  }
                                      }
                                 child{node {$\pmb{\frac{3}{4}}$}
                                             child{node {$\pmb{\frac{5}{8}}$}
                                                        child[dashed]{node {$\frac{2}{3}$} edge from parent [myzigzagedge]
                                                                            child[dashed]{node {$\frac{2}{3}+\frac{1}{\omega}$}
                                                                                         }
                                                                            child[dashed]{node {$\frac{2}{3}-\frac{1}{\omega}$}
                                                                                         }
                                                                     }
                                                  }
                                             child{node {$\pmb{\frac{7}{8}}$}
                                                  }
                                      }
                           }
                      child{node {$2$}
                                 child{node {$\frac{3}{2}$}
                                             child{node {$\frac{5}{4}$}
                                                  }
                                             child{node {$\frac{7}{4}$}
                                                  }
                                      }
                                 child{node {$3$}
                                             child{node {$\frac{5}{2}$}
                                                  }
                                             child{node {$4$}
                                                        child[dashed]{node {$\pi$} edge from parent [myzigzagedge]
                                                             }
                                                        child[dashed]{node {$\omega$}
                                                                            child[dashed]{node {$\omega-1$}
                                                                                         }
                                                                            child[dashed]{node {$\omega+1$}
                                                                                         }
                                                             }
                                                  }
                                      }
                           }
               };
\end{tikzpicture}
}
\caption{In the transfinite number tree, the nodes at level $\alpha$ correspond to the games born on day $\alpha$. The dyadics $x$ such that $|x|\leqslant 1$ are displayed in bold.}
\label{fig:fig2}
\end{center}
\end{figure}

Theorem \ref{th:brnumbers} provides a necessary result on birthdays. Although the proof is straightforward, it is included here since it does not appear concisely in the literature.
\vspace{0.3cm}
\begin{theorem} \label{th:brnumbers}
If $x$ is an integer then $b(x) = |x|$. If $x$ is a non-integer dyadic such that $|x|=m+\frac{r}{2^n}$, where $0< r<2^n$ and $r$ is odd, then $b(x)=m+1+n$.
\end{theorem}

\begin{proof}
Let us consider only $x\geqslant 0$, as the proof for $x<0$ is entirely analogous. If $x=0$, then $b(x)=0$ and the result is verified. If $x$ is a positive integer, then its canonical form is $\{x-1\,|\,\}$ and $b(x)=1+b(x-1)$. In this case as well, by induction, $b(x)=1+x-1=x$ and the result is verified. Finally, if $x=m+\frac{r}{2^n}>0$ where $0<r=2j+1<2^n$, and $j\geqslant 0$, then its canonical form is $\{m+\frac{j}{2^{n-1}}\,|\,m+\frac{j+1}{2^{n-1}}\}$, where the options appear in their reduced form. The options are integers when $j=0$ and $n=1$; in that case, $x=\{m\,|\,m+1\}$, $b(x)=1+m+1=1+m+n$, and the theorem holds. Otherwise, between $j$ and $j+1$, one of the numbers is even and the other is odd. Therefore, one of the fractions should reduce, but the other should not. The irreducible fraction is the fundamental option in terms of determining the birthday of $x$ and, by induction, its birthday is equal to $m+1+n-1$. Thus, $b(x)=1+m+1+n-1=m+1+n$ and the proof is complete.
\end{proof}

When playing a disjunctive sum of games, it is important to avoid playing on numbers. This is because numbers represent guaranteed moves that should only be used in the endgames. This assertion is codified in Theorem \ref{th:avoidance}, which is also necessary in this document.

\begin{theorem}[Number Avoidance Theorem, \cite{Sie013}, page 78] \label{th:avoidance} Suppose that $x$ is equal to a
number and $G$ is not, and let $H$ be an arbitrary game. If Left (resp. Right)
has a winning move on $G + H + x$, then she (he) has a winning move of the form
$(G + H)^L + x$ (resp. $(G + H)^R + x$).
\end{theorem}

When two players play a game, they eventually reach a position whose value is a number. Naturally, Left attempts to have this number be as large as possible, while Right wants it to be as small as possible. The number arrived at when Left moves first and plays optimally is called the Left Stop (LS), while the number reached when Right moves first and plays optimally is called the Right Stop (RS). The facts offered by Theorem \ref{th:stops} will be used later in this paper.

\begin{theorem}[\cite{Sie013}, pages 75 and 77] \label{th:stops}
Let $G$ and $H$ be short games. Let $x$ be a number.
\begin{enumerate}
  \item $\LSI{-G}=-\RSI{G}$ and $\RSI{-G}=-\LSI{G}$;
    \item $\LSI{G}\geqslant \RSI{G}$;
  \item $\LSI{G}<x\implies G<x$ and $\RSI{G}>x\implies G>x$;
    \item $\LSI{G+x}=\LSI{G}+x$ and $\RSI{G+x}=\RSI{G}+x$;
    \item $\LSI{G}+\RSI{H}\leqslant \LSI{G+H}\leqslant \LSI{G}+\LSI{H}$;
  \item $\RSI{G}+\RSI{H}\leqslant \RSI{G+H}\leqslant \LSI{G}+\RSI{H}$.
\end{enumerate}
\end{theorem}

A game $G$ is \emph{cold} if it is a number, \emph{tepid} if $\LS{G}=\RS{G}$ but $G$ is not a number, and \emph{hot} if $\LS{G}>\RS{G}$. When a game is tepid and $\LS{G}=\RS{G}=0$, it is an \emph{infinitesimal}. When a game is tepid and $\LS{G}=\RS{G}\neq 0$, it is a translation of an infinitesimal, that is, $G=x+\epsilon$, where $x$ is a nonzero number and $\epsilon$ is an infinitesimal. These are the possible ``natures'' of short games. Therefore, our proofs will go through these cases.

Nimbers are an important class of infinitesimals. The Sprague-Grundy theorem states that every impartial position played under normal play convention is equivalent to a single \textsc{nim} heap \citep{Gru39, Spra35}. As a result, each impartial position has a Grundy value, a nonnegative integer ${\cal G}(G)$, representing the size of the corresponding \textsc{nim} heap. The game value of an impartial position whose Grundy value is $n$ is the nimber $*n=\{0,\ldots,*(n-1)\,|\,0 ,\ldots,*(n-1)\}$. Regarding the disjunctive sum, we have $*n+*m=*(n\oplus m)$, where $+$ is the disjunctive sum and $\oplus$ is the \textsc{nim} sum, i.e., addition of the binary representations of $n$ and $m$ without carrying the one.

The structures $(\{0,\ldots,2^{k}-1\},\oplus)\cong \bigoplus_{i=1}^k\mathbb{Z}_2$ are finite groups (if $k=0$ we have the singleton group $\{0\}$). Table~\ref{tab:table} displays the algebra of the first five groups ($k\leqslant 4$). Considering binary representations, the \textsc{nim} sum never increases the number of digits (no carry), so, for any $n,m\in\{0,\ldots,2^{k}-1\}$, we have $n\oplus m<2^{k}$. That property leads to some important facts, expressed in Theorem \ref{th:nimsum}.

\begin{theorem}\label{th:nimsum} Let $j, k,n$ be nonnegative integers.
\begin{itemize}
\item $b(*j) = j$;
\item If $0\leqslant j\leqslant k$ and $*j+*k=*(2^n)$, then $k\geqslant 2^n$;
\item If $G+H=*(2^n)$ is a decomposition, then neither $G$ nor $H$ is a nimber.
\end{itemize}
\end{theorem}

\begin{proof} If $j=0$, then $b(j)=b(0)=0$. If $j\neq0$, then, due to the fact that the canonical form of $*j$ is $\{0,\ldots,*(j-1)\,|\,0,\ldots,*(j-1)\}$, we have\linebreak $b(*j)=1+\max\{b(0),\ldots,b(*(j-1))\}$. Therefore, by induction, we have $b(*j)=1+\max\{0,\ldots,j-1\}=j$.

Regarding item 2, note that when the nonnegative integers are written in binary, the \textsc{nim} sum never increases the number of digits. Thus, for any $j,k\in\{0,\ldots,2^{n}-1\}$, we have $j\oplus k<2^{n}$. Since that $*j+*k=*(2^n)$, either $j$ or $k$ must be greater than or equal to $2^n$. As the assumption is $0\leqslant j\leqslant k$, it follows that $k$ must be greater than or equal to $2^n$.

Finally, let $G+H=*(2^n)$ be a decomposition. If $G$ is a nimber, then $H$ is the sum of two nimbers, i.e., it is also a nimber. Thus, for some $j,k$, $G=*j$ and $H=*k$, and we may assume that $0\leqslant j\leqslant k$. By items 1 and 2, we have $k\geqslant 2^n$ and $b(*k)\geqslant b(*(2^n))$. Hence, $*j+*k = *(2^n)$ is not a decomposition, and that is a contradiction. Item 3 is also proved.
\end{proof}

 \begin{table}[!htbp]
\begin{center}
\newcommand\width{5}
\renewcommand\arraystretch{1.4}
\scalebox{0.8}{
\begin{tabular}{ccccccccccccccccc}
\hlineB{\width}
\multicolumn{1}{V{\width}c}{$\pmb{\oplus}$}&\multicolumn{1}{|cV{\width}}{$\pmb{0}$}&\multicolumn{1}{cV{\width}}{$\pmb{1}$}&\multicolumn{1}{c|}{$\pmb{2}$}&\multicolumn{1}{cV{\width}}{$\pmb{3}$}&\multicolumn{1}{c|}{$\pmb{4}$}&\multicolumn{1}{c|}{$\pmb{5}$}&\multicolumn{1}{c|}{$\pmb{6}$}&\multicolumn{1}{cV{\width}}{$\pmb{7}$}&\multicolumn{1}{c|}{$\pmb{8}$}&\multicolumn{1}{c|}{$\pmb{9}$}&\multicolumn{1}{c|}{$\pmb{10}$}&\multicolumn{1}{c|}{$\pmb{11}$}&\multicolumn{1}{c|}{$\pmb{12}$}&\multicolumn{1}{c|}{$\pmb{13}$}&\multicolumn{1}{c|}{$\pmb{14}$}&\multicolumn{1}{cV{\width}}{$\pmb{15}$}\\
\hline
\multicolumn{1}{V{\width}c}{$\pmb{0}$}&\multicolumn{1}{|cV{\width}}{$0$}&\multicolumn{1}{cV{\width}}{$1$}&\multicolumn{1}{c|}{$2$}&\multicolumn{1}{cV{\width}}{$3$}&\multicolumn{1}{c|}{$4$}&\multicolumn{1}{c|}{$5$}&\multicolumn{1}{c|}{$6$}&\multicolumn{1}{cV{\width}}{$7$}&\multicolumn{1}{c|}{$8$}&\multicolumn{1}{c|}{$9$}&\multicolumn{1}{c|}{$10$}&\multicolumn{1}{c|}{$11$}&\multicolumn{1}{c|}{$12$}&\multicolumn{1}{c|}{$13$}&\multicolumn{1}{c|}{$14$}&\multicolumn{1}{cV{\width}}{$15$}\\
\clineB{1-2}{\width}\clineB{3-17}{1}
\multicolumn{1}{V{\width}c}{$\pmb{1}$}&\multicolumn{1}{|c|}{$1$}&\multicolumn{1}{cV{\width}}{$0$}&\multicolumn{1}{c|}{$3$}&\multicolumn{1}{cV{\width}}{$2$}&\multicolumn{1}{c|}{$5$}&\multicolumn{1}{c|}{$4$}&\multicolumn{1}{c|}{$7$}&\multicolumn{1}{cV{\width}}{$6$}&\multicolumn{1}{c|}{$9$}&\multicolumn{1}{c|}{$8$}&\multicolumn{1}{c|}{$11$}&\multicolumn{1}{c|}{$10$}&\multicolumn{1}{c|}{$13$}&\multicolumn{1}{c|}{$12$}&\multicolumn{1}{c|}{$15$}&\multicolumn{1}{cV{\width}}{$14$}\\
\clineB{1-3}{\width}\clineB{4-17}{1}
\multicolumn{1}{V{\width}c}{$\pmb{2}$}&\multicolumn{1}{|c|}{$2$}&\multicolumn{1}{c|}{$3$}&\multicolumn{1}{c|}{$0$}&\multicolumn{1}{cV{\width}}{$1$}&\multicolumn{1}{c|}{$6$}&\multicolumn{1}{c|}{$7$}&\multicolumn{1}{c|}{$4$}&\multicolumn{1}{cV{\width}}{$5$}&\multicolumn{1}{c|}{$10$}&\multicolumn{1}{c|}{$11$}&\multicolumn{1}{c|}{$8$}&\multicolumn{1}{c|}{$9$}&\multicolumn{1}{c|}{$14$}&\multicolumn{1}{c|}{$15$}&\multicolumn{1}{c|}{$12$}&\multicolumn{1}{cV{\width}}{$13$}\\
\hline
\multicolumn{1}{V{\width}c}{$\pmb{3}$}&\multicolumn{1}{|c|}{$3$}&\multicolumn{1}{c|}{$2$}&\multicolumn{1}{c|}{$1$}&\multicolumn{1}{cV{\width}}{$0$}&\multicolumn{1}{c|}{$7$}&\multicolumn{1}{c|}{$6$}&\multicolumn{1}{c|}{$5$}&\multicolumn{1}{cV{\width}}{$4$}&\multicolumn{1}{c|}{$11$}&\multicolumn{1}{c|}{$10$}&\multicolumn{1}{c|}{$9$}&\multicolumn{1}{c|}{$8$}&\multicolumn{1}{c|}{$15$}&\multicolumn{1}{c|}{$14$}&\multicolumn{1}{c|}{$13$}&\multicolumn{1}{cV{\width}}{$12$}\\
\clineB{1-5}{\width}\clineB{6-17}{1}
\multicolumn{1}{V{\width}c}{$\pmb{4}$}&\multicolumn{1}{|c|}{$4$}&\multicolumn{1}{c|}{$5$}&\multicolumn{1}{c|}{$6$}&\multicolumn{1}{c|}{$7$}&\multicolumn{1}{c|}{$0$}&\multicolumn{1}{c|}{$1$}&\multicolumn{1}{c|}{$2$}&\multicolumn{1}{cV{\width}}{$3$}&\multicolumn{1}{c|}{$12$}&\multicolumn{1}{c|}{$13$}&\multicolumn{1}{c|}{$14$}&\multicolumn{1}{c|}{$15$}&\multicolumn{1}{c|}{$8$}&\multicolumn{1}{c|}{$9$}&\multicolumn{1}{c|}{$10$}&\multicolumn{1}{cV{\width}}{$11$}\\
\hline
\multicolumn{1}{V{\width}c}{$\pmb{5}$}&\multicolumn{1}{|c|}{$5$}&\multicolumn{1}{c|}{$4$}&\multicolumn{1}{c|}{$7$}&\multicolumn{1}{c|}{$6$}&\multicolumn{1}{c|}{$1$}&\multicolumn{1}{c|}{$0$}&\multicolumn{1}{c|}{$3$}&\multicolumn{1}{cV{\width}}{$2$}&\multicolumn{1}{c|}{$13$}&\multicolumn{1}{c|}{$12$}&\multicolumn{1}{c|}{$15$}&\multicolumn{1}{c|}{$14$}&\multicolumn{1}{c|}{$9$}&\multicolumn{1}{c|}{$8$}&\multicolumn{1}{c|}{$11$}&\multicolumn{1}{cV{\width}}{$10$}\\
\hline
\multicolumn{1}{V{\width}c}{$\pmb{6}$}&\multicolumn{1}{|c|}{$6$}&\multicolumn{1}{c|}{$7$}&\multicolumn{1}{c|}{$4$}&\multicolumn{1}{c|}{$5$}&\multicolumn{1}{c|}{$2$}&\multicolumn{1}{c|}{$3$}&\multicolumn{1}{c|}{$0$}&\multicolumn{1}{cV{\width}}{$1$}&\multicolumn{1}{c|}{$14$}&\multicolumn{1}{c|}{$15$}&\multicolumn{1}{c|}{$12$}&\multicolumn{1}{c|}{$13$}&\multicolumn{1}{c|}{$10$}&\multicolumn{1}{c|}{$11$}&\multicolumn{1}{c|}{$8$}&\multicolumn{1}{cV{\width}}{$9$}\\
\hline
\multicolumn{1}{V{\width}c}{$\pmb{7}$}&\multicolumn{1}{|c|}{$7$}&\multicolumn{1}{c|}{$6$}&\multicolumn{1}{c|}{$5$}&\multicolumn{1}{c|}{$4$}&\multicolumn{1}{c|}{$3$}&\multicolumn{1}{c|}{$2$}&\multicolumn{1}{c|}{$1$}&\multicolumn{1}{cV{\width}}{$0$}&\multicolumn{1}{c|}{$15$}&\multicolumn{1}{c|}{$14$}&\multicolumn{1}{c|}{$13$}&\multicolumn{1}{c|}{$12$}&\multicolumn{1}{c|}{$11$}&\multicolumn{1}{c|}{$10$}&\multicolumn{1}{c|}{$9$}&\multicolumn{1}{cV{\width}}{$8$}\\
\clineB{1-9}{\width}\clineB{10-17}{1}
\multicolumn{1}{V{\width}c}{$\pmb{8}$}&\multicolumn{1}{|c|}{$8$}&\multicolumn{1}{c|}{$9$}&\multicolumn{1}{c|}{$10$}&\multicolumn{1}{c|}{$11$}&\multicolumn{1}{c|}{$12$}&\multicolumn{1}{c|}{$13$}&\multicolumn{1}{c|}{$14$}&\multicolumn{1}{c|}{$15$}&\multicolumn{1}{c|}{$0$}&\multicolumn{1}{c|}{$1$}&\multicolumn{1}{c|}{$2$}&\multicolumn{1}{c|}{$3$}&\multicolumn{1}{c|}{$4$}&\multicolumn{1}{c|}{$5$}&\multicolumn{1}{c|}{$6$}&\multicolumn{1}{cV{\width}}{$7$}\\
\hline
\multicolumn{1}{V{\width}c}{$\pmb{9}$}&\multicolumn{1}{|c|}{$9$}&\multicolumn{1}{c|}{$8$}&\multicolumn{1}{c|}{$11$}&\multicolumn{1}{c|}{$10$}&\multicolumn{1}{c|}{$13$}&\multicolumn{1}{c|}{$12$}&\multicolumn{1}{c|}{$15$}&\multicolumn{1}{c|}{$14$}&\multicolumn{1}{c|}{$1$}&\multicolumn{1}{c|}{$0$}&\multicolumn{1}{c|}{$3$}&\multicolumn{1}{c|}{$2$}&\multicolumn{1}{c|}{$5$}&\multicolumn{1}{c|}{$4$}&\multicolumn{1}{c|}{$7$}&\multicolumn{1}{cV{\width}}{$6$}\\
\hline
\multicolumn{1}{V{\width}c}{$\pmb{10}$}&\multicolumn{1}{|c|}{$10$}&\multicolumn{1}{c|}{$11$}&\multicolumn{1}{c|}{$8$}&\multicolumn{1}{c|}{$9$}&\multicolumn{1}{c|}{$14$}&\multicolumn{1}{c|}{$15$}&\multicolumn{1}{c|}{$12$}&\multicolumn{1}{c|}{$13$}&\multicolumn{1}{c|}{$2$}&\multicolumn{1}{c|}{$3$}&\multicolumn{1}{c|}{$0$}&\multicolumn{1}{c|}{$1$}&\multicolumn{1}{c|}{$6$}&\multicolumn{1}{c|}{$7$}&\multicolumn{1}{c|}{$4$}&\multicolumn{1}{cV{\width}}{$5$}\\
\hline
\multicolumn{1}{V{\width}c}{$\pmb{11}$}&\multicolumn{1}{|c|}{$11$}&\multicolumn{1}{c|}{$10$}&\multicolumn{1}{c|}{$9$}&\multicolumn{1}{c|}{$8$}&\multicolumn{1}{c|}{$15$}&\multicolumn{1}{c|}{$14$}&\multicolumn{1}{c|}{$13$}&\multicolumn{1}{c|}{$12$}&\multicolumn{1}{c|}{$3$}&\multicolumn{1}{c|}{$2$}&\multicolumn{1}{c|}{$1$}&\multicolumn{1}{c|}{$0$}&\multicolumn{1}{c|}{$7$}&\multicolumn{1}{c|}{$6$}&\multicolumn{1}{c|}{$5$}&\multicolumn{1}{cV{\width}}{$4$}\\
\hline
\multicolumn{1}{V{\width}c}{$\pmb{12}$}&\multicolumn{1}{|c|}{$12$}&\multicolumn{1}{c|}{$13$}&\multicolumn{1}{c|}{$14$}&\multicolumn{1}{c|}{$15$}&\multicolumn{1}{c|}{$8$}&\multicolumn{1}{c|}{$9$}&\multicolumn{1}{c|}{$10$}&\multicolumn{1}{c|}{$11$}&\multicolumn{1}{c|}{$4$}&\multicolumn{1}{c|}{$5$}&\multicolumn{1}{c|}{$6$}&\multicolumn{1}{c|}{$7$}&\multicolumn{1}{c|}{$0$}&\multicolumn{1}{c|}{$1$}&\multicolumn{1}{c|}{$2$}&\multicolumn{1}{cV{\width}}{$3$}\\
\hline
\multicolumn{1}{V{\width}c}{$\pmb{13}$}&\multicolumn{1}{|c|}{$13$}&\multicolumn{1}{c|}{$12$}&\multicolumn{1}{c|}{$15$}&\multicolumn{1}{c|}{$14$}&\multicolumn{1}{c|}{$9$}&\multicolumn{1}{c|}{$8$}&\multicolumn{1}{c|}{$11$}&\multicolumn{1}{c|}{$10$}&\multicolumn{1}{c|}{$5$}&\multicolumn{1}{c|}{$4$}&\multicolumn{1}{c|}{$7$}&\multicolumn{1}{c|}{$6$}&\multicolumn{1}{c|}{$1$}&\multicolumn{1}{c|}{$0$}&\multicolumn{1}{c|}{$3$}&\multicolumn{1}{cV{\width}}{$2$}\\
\hline
\multicolumn{1}{V{\width}c}{$\pmb{14}$}&\multicolumn{1}{|c|}{$14$}&\multicolumn{1}{c|}{$15$}&\multicolumn{1}{c|}{$12$}&\multicolumn{1}{c|}{$13$}&\multicolumn{1}{c|}{$10$}&\multicolumn{1}{c|}{$11$}&\multicolumn{1}{c|}{$8$}&\multicolumn{1}{c|}{$9$}&\multicolumn{1}{c|}{$6$}&\multicolumn{1}{c|}{$7$}&\multicolumn{1}{c|}{$4$}&\multicolumn{1}{c|}{$5$}&\multicolumn{1}{c|}{$2$}&\multicolumn{1}{c|}{$3$}&\multicolumn{1}{c|}{$0$}&\multicolumn{1}{cV{\width}}{$1$}\\
\hline
\multicolumn{1}{V{\width}c}{$\pmb{15}$}&\multicolumn{1}{|c|}{$15$}&\multicolumn{1}{c|}{$14$}&\multicolumn{1}{c|}{$13$}&\multicolumn{1}{c|}{$12$}&\multicolumn{1}{c|}{$11$}&\multicolumn{1}{c|}{$10$}&\multicolumn{1}{c|}{$9$}&\multicolumn{1}{c|}{$8$}&\multicolumn{1}{c|}{$7$}&\multicolumn{1}{c|}{$6$}&\multicolumn{1}{c|}{$5$}&\multicolumn{1}{c|}{$4$}&\multicolumn{1}{c|}{$3$}&\multicolumn{1}{c|}{$2$}&\multicolumn{1}{c|}{$1$}&\multicolumn{1}{cV{\width}}{$0$}\\
\hlineB{\width}
\end{tabular}
}
\caption{Finite groups $(\{0,\ldots,2^{k}-1\},\oplus), k\leqslant 4$.}
\label{tab:table}
\end{center}
\end{table}

In the main proofs of this document, we will have to analyze tepid components, that is, translations of infinitesimals. In that process, we will have to relate the birthdays and followers of these tepid components with the birthdays and followers of the related infinitesimals. Theorem \ref{th:translationrev}, Corollary \ref{th:birthnimber}, and Lemma \ref{lem:remoteness} all deal with this. These results are a direct consequence of the Number Translation Principle. Once again, although the proofs are straightforward, they are included here since they do not appear concisely in the literature.

\begin{theorem}[Number Translation Principle, \cite{Sie013}, page 78]\label{th:translation}
Suppose $x$ is equal to a number and $G$ is not. Then,
\[G+x=\{G^\mathcal{L}+x\,|\,G^\mathcal{R}+x\}.\]
\end{theorem}

\begin{theorem}[Number Translation Principle Revisited]\label{th:translationrev}
Suppose that $x$ is a number, whereas the canonical form $G$ is not. Then, we can state the following:
\begin{enumerate}
  \item The canonical form of $x+G$ is $\{x+G^\mathcal{L}\,|\,x+G^\mathcal{R}\}$, where the elements of $x+G^\mathcal{L}$ and $x+G^\mathcal{R}$ are in canonical form;
  \item If $G'$ is a follower of $G$ that is not a number, then the canonical form of $x+G'$ is a follower of the canonical form of $x+G$.
\end{enumerate}
\end{theorem}

\begin{proof}
The first item is a direct consequence of the Number Translation Principle. Regarding the second item, if $G' = G$, the result is verified since a game is a follower of itself. Hence, let us assume first that $G' \in G^\mathcal{L}\cup G^\mathcal{R}$. The first item implies that the canonical form of $x + G'$ is an option of the canonical form of $x+G$, and the result is also verified. On the other hand, if $G' \not\in G^\mathcal{L}\cup G^\mathcal{R}$, then $G'$ must be a follower of some $G^L$ or $G^R$  that is not a number. By induction, we know that the canonical form of $x + G'$ is a follower of the canonical form of $x + G^L$ or $x + G^R$. Thus, the canonical form of $x+G'$ is a follower of the canonical form of $x+G$. The proof is complete.
\end{proof}

\begin{corollary}\label{th:birthnimber}
If $x$ is a number and $*n$ is a nimber, then $b(x+*n)=b(x)+ n$.
\end{corollary}

\begin{proof}
By Theorem \ref{th:translationrev}, the canonical form of $x+G$ is $\{x+G^\mathcal{L}\,|\,x+G^\mathcal{R}\}$, where the elements of $x+G^\mathcal{L}$ and $x+G^\mathcal{R}$ are in canonical form. Thus, $b(x+*n)=1+\max\{b(x+0),\ldots,b(x+*(n-1))\}$, and, by induction, $b(x+*n)=1+\max\{b(x)+0,\ldots,b(x)+n-1\}=b(x)+n$.
\end{proof}

\begin{observation}
In general, it is possible to have $b(x+G)<b(G)$. For example, this happens if x=$\frac{1}{2}$ and $G=\pm \frac{1}{2}$.
\end{observation}

Another useful fact is the notion of remoteness. We shall see that as long as $n$ is big enough the exact $*n$ is immaterial with respect to certain disjunctive sums. With respect to decompositions, Theorem \ref{th:remoteness} and Lemma \ref{lem:remoteness} will allow us to establish useful connections between components and nimbers.

\begin{definition}\label{def:remoteness}
A nimber $*n$ is remote for $G$ if it is not equal to any
follower of $G$.
\end{definition}

\begin{theorem}[Norton Remoteness Theorem, \cite{Sie013}, page 138]\label{th:remoteness}Suppose $*n_0$ is remote for $G$. Then, for all $n>n_0$,
$o(G+*n) = o(G+*n_0)$.
\end{theorem}

\begin{lemma}\label{lem:remoteness}Let $G'$ be an infinitesimal that is not a nimber and is in canonical form. Let $x$ be a number, and let $G$ be the canonical form of the tepid game $x+G'$. If $b(G)< k$ then $*(k-1)$ is remote for $G'$.
\end{lemma}

\begin{proof}
Suppose that $*(k-1)$ is not remote for $G'$. By Theorem~\ref{th:translationrev}, $x+*(k-1)$ is a proper follower of $G$, and, by Corollary \ref{th:birthnimber}, $b(x+*(k-1))=b(x)+k-1$. However, this equality contradicts the inequality $b(G)<k$.
\end{proof}

When analyzing disjunctive sums of hot games, it is natural to anchor the first level of analysis in stops, a concept that recurs through all followers of games. In the case of disjunctive sums of infinitesimals, our paper introduces the concept of \emph{distance to a nimber or better}, which is developed in Section~\ref{sec:dist}. Distance recurs also through all followers of games. As we will see, this concept works very well when used with Norton Remoteness Theorem. Together, they provide an original type of proof. In a way, this demonstrative construction is another contribution of this document.

\section{Indecomposable numbers}
\label{sec:numbers}
 We can observe that dyadics in Figure \ref{fig:fig2} that are not between $-1$ and $1$ admit natural decompositions. For example,  $\frac{7}{4}=1\frac{3}{4}=1+\frac{3}{4}$. Since $b(\frac{7}{4})=4$, $b(1)=1$, and $b(\frac{3}{4})=3$, the last sum is a strong decomposition. To transform this observation into a proof, we have to analyze disjunctive sums of the type $G+H+x$, where $x$ is a dyadic. It is important to mention that we have to consider \emph{all possible game forms} $G$ and $H$. These may be numbers, tepid forms or hot forms. Furthermore, it is important to argue that moves on $x$ are ``bad moves'' and do not need to be considered. For that, we use Theorem~\ref{th:avoidance}.

\begin{theorem}\label{th:numbers}A number $x$ is indecomposable if and only if $|x|\leqslant 1$.
\end{theorem}

\begin{proof} Let $|x|=i+\frac{r}{2^n}$, where $i$ is a nonnegative integer and the integer $0\leqslant r<2^n$, if positive, is odd. Observe that $x=i+\frac{r}{2^n}$ or $x=-i-\frac{r}{2^n}$, depending on the sign of $x$. In this proof, we will assume that $x>0$ since the arguments for $x<0$ are similar.\\

\noindent
($\Rightarrow$) Suppose $x>1$. If $x$  is an integer then $x=1+(x-1)$  is a decomposition of $x$. If $x$ is not an integer, then $x=i+\frac{r}{2^n}$  is a decomposition of $x$ since
$b(i)=i$, $b(\frac{r}{2^n})=n+1$, and $b(x)=i+1+n$ (Theorem \ref{th:brnumbers}).  \\

\noindent
($\Leftarrow$) Suppose $x=\frac{r}{2^n}$, and $0\leqslant r\leqslant 2^n$. If $x=0$ then $x$ is trivially indecomposable since there are no summands with birthdays less than zero. If $x=1$  then $x$ is indecomposable since the only available summand is zero. Hence, consider the case, $0<x<1$. By Theorem \ref{th:brnumbers}, we know that $b(x)=n+1$. Suppose that $x=G+H$ is a decomposition of $x$,  that $G$ and $H$ are canonical forms, and that $b(G)+b(H)$ is minimum.\\

Suppose that $G$ is not a number. In that case, since $G=x-H$, then $H$ cannot be a number either. Also, since $G$ and $H$ are not numbers, both players have options in $G$ and $H$. We are assuming that $G+H-x=0$, therefore Left has a winning response to the Right move $G^{R_1}+H-x$.  By Theorem \ref{th:avoidance}, there must be a winning move in $G^{R_1}$ or in $H$.
\begin{enumerate}
  \item[] If $G^{R_1L}+H-x>0$ (the inequality is strict because $b(G)+b(H)$ is minimum), then
 $G^{R_1L}+H-x>0 =G+H-x$, which yields, $G^{R_1L}>G$. This implies that $G^{R_1}$ is a reversible option contradicting that $G$ is in canonical form.

  \item[] If  $G^{R_1}+H^L-x>0$, then $G^{R_1}+H^L-x>0 =G+H-x$, and thus
  $G^{R_1}+H^L > G+H$. Now in $G+H^L-x$, Right has a winning move in either $G$ or $H^L$. By the argument in the previous paragraph, it must be to  $G^{R_2}+H^L-x<0$. Now we have
  $G^{R_1}+H^L-x>0>G^{R_2}+H^L-x$ which gives $G^{R_1}>G^{R_2}$. Thus, $G^{R_1}$ is a dominated option again contradicting that $G$ is in canonical form.
  \end{enumerate}

The only possible case is that both $G$ and $H$ are numbers. Since $G+H=x$, at least one of the summands has a denominator $2^j$ with $j\geqslant n$. Without loss of generality, say that is $G$. Since $G$ cannot be zero, by Theorem \ref{th:brnumbers}, $b(G)\geqslant j+1\geqslant n+1=b(x)$. That contradicts the fact that $b(G)$ is smaller than $b(x)$.\\
\end{proof}

\begin{theorem}\label{th:strongnumbers}Let $x$ be a number. If $|x|>1$ then $x$ is strongly decomposable and the only strong decompositions of $x$ are sums of numbers.
\end{theorem}

\begin{proof}
All strong decompositions $x=G+H$ are minimal in terms of\linebreak $b(G)+b(H)$. Therefore, all contradictions found in the proof of Theorem \ref{th:numbers} can be used, that is, $G$ and $H$ must be numbers. The only difference is that the last inequality cannot be deduced anymore. Indeed, there are strong decompositions of $x$, as the decompositions mentioned in the first implication of the same proof.
\end{proof}

\section{Distance to a nimber or better}
\label{sec:dist}
The proof of Theorem \ref{th:numbers} was based on assuming that $G+H-x=0$, then arguing that against a Right move in $G+H-x$ to $G^{R_1}+H-x$, Left did not have a winning move, which is a contradiction. In particular, a Left reply $G^{R_1L}+H-x$ was ruled out by reversibility, $G^{R_1}+H^L-x$ by domination, and $G^{R_1}+H-x^R$ by the strong version of Number Avoidance Theorem.

To find the indecomposable nimbers, we will need to analyze disjunctive sums of the form $G+H+*(2^n)$ and consider all of Left's responses to a Right move. Unlike the proof of Theorem \ref{th:numbers}, we cannot use the Number Avoidance Theorem to eliminate a Left response in $*(2^n)$. This fact means that the proof has to be more subtle.

 When $G$ and $H$ are hot forms, it will be possible make use of Left and Right stops to deal with a hypothetical Left answer to $G^{R_1}+H+*(2^n-k)$. The hardest problem occurs when both $G$ and $H$ are infinitesimal. One idea would be to use the Atomic Weight Theory. For example, suppose that $aw(G^{R_1}+H)=-1$. If Left plays on the nimber to $G^{R_1}+H+*(2^n-k)$ then, by the `two-ahead' rule, Right wins. That happens due to the fact that $aw(G^{R_1}+H+*(2^n-k))$ continues to be equal to $-1$ and it is Right's turn.
 In that example, Right's move to $G^{R_1}+H+*(2^n)$ creates a kind of ``race'' in the first two components in which Left cannot afford to ignore. Unfortunately, atomic weight is not defined for all infinitesimals, it is defined only for atomic games \citep{Sie013}. For example, the game $\{\{1\,|\,\cgdoubleup\}\,|\,0\}$ is infinitesimal and does not have an atomic weight. Left's move to $\{1\,|\,\cgdoubleup\}$ behaves like an ``infinitely large threat''; Right is compelled to respond before he may consider moving in a dicotic component.

Observe that, when $G$ is an infinitesimal, $\LS{G}=\RS{G}=0$. That implies that if Right plays first, even if Left has permission to pass, Right can force the situation where one of the players moves to a follower $G'\leqslant *m$, for some $m$. We will introduce a concept that focuses on the number of moves Right needs to achieve a nimber or better. This new concept is defined for \emph{all} infinitesimals and is applicable to \emph{certain}, but not all, disjunctive sums. This class of sums is rich enough for us to prove the indecomposability result.\\

\noindent
\textbf{Notation}: Let $G$ be an infinitesimal. If $G\cggfuz \cgstar k$ (or $G> \cgstar k$) for every nonnegative integer $k$, we write $G\cggfuz\cgallstar$ (or $G>\cgallstar$).
If $G\cglfuz *k$ (or $G< *k$) for every nonnegative integer $k$, we write $G\cglfuz\cgallstar$ (or $G<\cgallstar$).\\

\begin{definition} \label{def:dist}
Let $G$ be a game such that $\RSI{G}\leqslant 0$. The \emph{Right distance to a nimber or better} in $G$, $\RDI{G}$, is defined recursively:

\begin{enumerate}
  \item If there is $k$ such that $G\leqslant *k$, then $\RDI{G}=0$;
  \item Otherwise, \begin{align*}
&\RDI{G}=1+\min_{\substack{G^R\in G^\mathcal{R}: \\ \LSI{G^R}\leqslant 0}}\left\{\max\left\{\RDI{G^R},-1+\max_{G^{RL}\in G^{R\mathcal{L}}}\left\{\RDI{G^{RL}}\right\}\right\}\right\}.
\end{align*}
\end{enumerate}
\end{definition}

We will provide some intuition about the underlying idea of this concept. Before that, given that the second item of the definition is especially intricate, it is important to prove that this definition can indeed be made.

\begin{theorem}\label{th:welldefined}
If $G$ is a game form such that $\RSI{G}\leqslant 0$, then $\RDI{G}$ is well-defined.
\end{theorem}

\begin{proof}
Since $\RS{G}\leqslant 0$, either $G$ is zero, or $G$ has at least one Right option. If $G=0$, then, since $0$ is a nimber, the definition gives $\RD{G}=0$.

Suppose that $G$ has at least one Right option. If there is $k$ such that $G\leqslant *k$, then the definition gives again $\RD{G}=0$. Otherwise, by the definition of stops, there is some $G^R$ with $\LS{G^R}\leqslant 0$. From $\LS{G^R}\leqslant 0$, we can conclude two relevant inequalities: (1) $\RS{G^R}\leqslant\LS{G^R}\leqslant 0$ and (2) for all $G^{RL}\in G^{R\mathcal{L}}$, $\RS{G^{RL}}\leqslant \LS{G^R}\leqslant 0$. Now, by induction, both $\RD{G^R}$ and $\max\limits_{\substack{G^{RL}\in G^{R\mathcal{L}}}} \{\RD{G^{RL}}\}$ are well-defined. Thus, since the elements of the non-empty set to be minimized are well-defined, $\RD{G}$ is also\linebreak well-defined.
\end{proof}

\begin{observation}
If $\LSI{G}\geqslant 0$, then the definition of \emph{Left distance to a nimber or better} in  $G$, $\LDI{G}$, is defined in a similar way. Note that both $\RDI{G}$ and $\LDI{G}$ are nonnegative integers. Clearly, $\RDI{G}$ and $\LDI{G}$ are well-defined for all infinitesimals as, in such cases, $\LSI{G}$ and $\RSI{G}$ both equal zero. In the following proofs, whenever we write ``optimal option'', we are referring to an option that minimizes the set pointed out in the second item of Definition \ref{def:dist}.
\end{observation}

The following results and observations are only stated for Right distances, but the corresponding statements for Left distances also hold.\\

The Right distance is a \textit{worst-case} scenario. It is the largest number of unanswered moves that Left can force Right to make in order to get a follower that is less or equal to a nimber. Note, Left could be playing in some other component of a disjunctive sum that includes $G$.

In this definition, the base case of the recursion is when $G$ is already less or equal than a nimber. Then,
$\RD{G}=0$ and Right does not need to make moves to achieve a nimber or better. Note that since $\RS{G}=0$ and $k=0$ is allowed, then either $G$ or some follower of $G$ will satisfy this condition.

In part 2 of the definition, Right needs to make a move in $G$, and Right is trying to minimize the distance to his goal, hence the initial $+1$. After a Right move,
Left is trying to maximize the distance to Right's goal, and so has two possibilities. She can make no move, giving the $\RD{G^R}$ term. She can also move in order to delay Right's goal, playing the maximum, but she has also answered Right's move, which gives the $-1+\RD{G^{RL}}$ term. Left will choose the maximum of these two. Right will choose an optimal option, i.e., least in terms of distance, hence the initial minimization.

\begin{observation}\label{obs:remote}
A consequence of Definition \ref{def:dist} is the fact \[\RDI{G}\geqslant1 \implies G\cggfuz\cgallstar.\] Of course, if we had some $k$ such that $G\leqslant *k$, we would have $\RDI{G}=0$, and not $\RDI{G}\geqslant1$. Analogously, a consequence of Definition \ref{def:dist} is the fact \[\RDI{G}\geqslant2 \implies G>\cgallstar.\] If we had some $k$ such that $G\cgfuzzy *k$, we would have some $G^R$ such that $G^R\leqslant *k$ or some $j$ such that $G\leqslant *(k-j)$. In the first case, we would have $\RDI{G}\leqslant 1$; in the second case, we would have $\RDI{G}=0$.
\end{observation}

\noindent
\textbf{Examples:}
\begin{enumerate}
  \item  $\RD{\cgupstar}=1$;
  \item  $\RD{\cgdoubleup}=2$;
  \item  $\RD{\{0\,|\,\{0\,|\,-1\}\}}=1$;
  \item  $\RD{\{\{1\,|\,\cgdoubleup\}\,|\,0\}}=1$ (this game is not atomic);
  \item  $\RD{\{5.\cgup\,|\,\{3.\cgup\,|\,0\}\}}=3$ (in this game $\RD{G^R}=1$, but Left delays Right's goal with her answer).
\end{enumerate}
The next lemmas are almost immediate from the definition but will be used later.

\begin{lemma}\label{lem:DistanceSteps} Let $G$ be an infinitesimal and $\RDI{G}>0$. For an optimal $G^R$ in Definition \ref{def:dist},  $\RDI{G}>\RDI{G^R}$. If $\RDI{G}-1>\RDI{G^R}$, then there exists a $G^{RL}$ with $\RDI{G^{RL}}=\RDI{G}$, and
$\RDI{G^{RL'}}\leqslant\RDI{G}$ for any other option $G^{RL'}$.
\end{lemma}
\begin{proof}
Let $G^R$ be a Right optimal option as defined in Definition \ref{def:dist}.

First, consider the case where the best for Left is to ignore the alternating condition and not make any moves. By definition, this occurs when, for every $G^{RL}$, we have that $\RD{G^{RL}}-1<\RD{G^R}$. In this case, we have $\RD{G}=1+\RD{G^R}$, and the lemma is thus verified.

Now, suppose that Left can choose an optimal $G^{RL}$ as defined in Definition \ref{def:dist}. In this case, we have $\RD{G}=1+\RD{G^{RL}}-1=\RD{G^{RL}}$. Due to the maximality, we have $\RD{G^{RL}}-1\geqslant \RD{G^R}$. Combining these two facts, we get $\RD{G}=\RD{G^{RL}}\geqslant \RD{G^R}+1> \RD{G^R}$. It is only when Left makes a move that we can have $\RD{G}>\RD{G^R}+1$. However, there cannot be another option $G^{RL'}$ such that $\RD{G^{RL'}}>\RD{G}$. If such an option existed, since $\RD{G}=\RD{G^{RL}}$, then we would have $\RD{G^{RL'}}-1>\RD{G^{RL}}-1$, and $G^{RL}$ would not be optimal. Once again, the lemma is verified.
\end{proof}

The next definition gives a class of disjunctive sums required to characterize the indecomposable nimbers.

\begin{definition} \label{def:ssyst}
A \emph{star system} is a disjunctive sum $G+H+*k$ where\linebreak $\RSI{G}\leqslant0$ and $\LSI{H}\geqslant0$.
\end{definition}

Star systems have a rule similar to one of the cases in the two-ahead rule from Atomic Weight Theory.

\begin{lemma}\label{lem:DistanceZero} Let $G+H+*k$ be a star system. If $\RDI{G}=0$ and $\LDI{H}=1$ then Right wins playing first. If $\RDI{G}=0$ and $\LDI{H}\geqslant 2$ then Right wins playing first or second.
\end{lemma}
\begin{proof}
Starting with the first implication, the definition of distance implies that there is a nonnegative $k'$ such that $G\leqslant *k'.$ Therefore, we have that $G+H+*k\leqslant H+*k'+*k$. On the other hand, since $\LD{H}= 1$, we have $H\cglfuz\cgallstar$. Hence, there is a winning move for Right in $H+*k'+*k$, and, consequently, there is a winning move for Right in $G+H+*k$.

Regarding the second implication, if Right plays first in $G+H+*k$, the argument is completely analogous to the one used to prove the first implication. Therefore, let us assume that Left plays first in $G+H+*k$. If Left moves to $G+H+*(k-j)$, then it is Right's turn, and we are in the case that was previously mentioned. If Left makes a move to $G+H^L+*k$ where $\RS{H^L}<0$, while $\LS{G}\leqslant0$ (due to $\RD{G}=0$) and $\LS{*k}=0$, then, by Theorem~\ref{th:stops}, we have $\RS{G+H^L+*k}\leqslant\LS{G}+\RS{H^L+*k}\leqslant\RS{H^L+*k}\leqslant \RS{H^L}+\LS{*k}<0$. This means that Right wins. If Left moves to $G+H^L+*k$ where $\RS{H^L}\geqslant0$, without reducing the distance with $H^L$ by more than one unit, then we either fall into the first case, or, by induction, Right wins. If Left moves to $G+H^L+*k$ where $\RS{H^L}\geqslant0$, and reduces the distance by more than one unit with $H^L$, then, according to Lemma~\ref{lem:DistanceSteps}, Right has an answer to $G+H^{LR}+*k$ where $\LD{H^{LR}}\geqslant 2$. Also in that case, by induction, Right wins. Finally, if Left moves to $G^L+H+*k$, two things can happen. Since $\RD{G}=0$, we know that $G\leqslant *k'$. This means that either there exists $G^{LR}+*k'\leqslant 0$, or there exists $G^{L}+*(k'-j)\leqslant 0$. In the first case, we have $G^{LR}+H+*k\leqslant H+*k'+*k<0$, and Right wins. In the second case, we have $G^L+H+*k\leqslant H+*(k'-j)+*k<0$, and Right also wins. The strict inequalities are due to the fact that $\LD{H}\geqslant 2$ (Observation \ref{obs:remote}).
\end{proof}

\begin{theorem}\label{th:twoahead}
Let $G+H+*k$ be a star system. If $\LDI{H}-\RDI{G}\geqslant 1$, then Right has a winning move in $G+H+*k$. Left has a winning move if  $\RDI{H}-\LDI{G}\geqslant 1$.
\end{theorem}

\begin{proof}We only prove the first part since the proof of the second is similar.

Suppose first that $\RD{G}=0$ and $\LD{H}\geqslant 1$. By Lemma \ref{lem:DistanceZero}, there is a winning move for Right in $G+H+*k$.

Suppose now that $\RD{G}\geqslant1$ and $\LD{H}\geqslant2$. Let us see that Right wins by moving to $G^R+H+*k$, where
$G^R$ is an optimal option as defined in Definition~\ref{def:dist}. After this optimal move, the difference $\LD{H}-\RD{G^R}$ is at least two. If $\RD{G^R}=0$ then, by Lemma \ref{lem:DistanceZero}, Right wins. Hence, suppose that $\RD{G^R}$ is still positive. If Left moves to $G^R+H^L+*k$ with $\RS{H^L}<0$, while $\LS{G^R}\leqslant 0$ and $\LS{*k}=0$, then, by Theorem~\ref{th:stops},
we have $\RS{G^R+H^L+*k}\leqslant\LS{G^R}+\RS{H^L+*k}\leqslant\RS{H^L+*k}\leqslant \RS{H^L}+\LS{*k}<0$.
This means that Right wins. Otherwise, if Left answers to $G^R+H^L+*k$ and $\LD{H^L}-\RD{G^R}< 1$, then
we have the inequality $\LD{H^L}<\LD{H}-1$. In accordance with Lemma \ref{lem:DistanceSteps}, there exists a Right option $H^{LR}$ with $\LD{H^{LR}}=\LD{H}$, meaning that the Left distance is preserved in the second component. Therefore, given that we are considering short games, there will come a point in some $H'$ that is a follower of $H$ where Left will no longer be able to respond to $G^R+H'^L+*k$ and $\LD{H'^L}-\RD{G^R}< 1$. Regarding that moment, we consider the following possibilities.

If Left plays to $G^R+H'+*(k-j)$, then $\LD{H'}-\RD{G^R}\geqslant 2$, and, by induction, Right wins.

If Left plays to $G^R+H'^L+*k$ where $\LD{H'^L}-\RD{G^R}\geqslant 1$, then, by induction, Right wins.

If  Left answers to $G^{RL}+H'+*k$, again by Lemma \ref{lem:DistanceSteps}, $\RD{G^{RL}}\leqslant \RD{G}$. In that case, we have $\LD{H'}-\RD{G^{RL}}\geqslant \LD{H'}-\RD{G}\geqslant 1$, and, by induction, Right wins.
\end{proof}

\begin{corollary}\label{th:twoahead2}
Let  $G$ and $H$ be two infinitesimals. If $G+H+*k=0$ then $\LDI{H}=\RDI{G}$ and $\LDI{G}=\RDI{H}$.
\end{corollary}

\begin{proof}
If $\LD{H}>\RD{G}$ then we have $\LD{H}-\RD{G}\geqslant 1$ and, by Theorem \ref{th:twoahead}, Right has a winning move. That contradicts the assumption\linebreak $G+H+*k\in\mathcal{P}$. The other inequalities lead to a similar contradictions. \end{proof}

\section{Indecomposable nimbers}
\label{sec:nimbers}
In Table~\ref{tab:table}, we can observe that if the sizes of nimbers are not powers of $2$, we have natural decompositions. However, if the sizes are powers of $2$, that is not the case. For example, consider $*3=*+*2$. Since $b(*2) = 2$, $b(*) = 1$, and $b(*3) = 3$, the last disjunctive sum is a strong decomposition. To transform this observation into a proof, we have to analyze decompositions of the type $G+H+*(2^n)$. It is important to mention that we have to consider \emph{all possible game forms} $G$ and $H$. These may be numbers, tepid forms, or hot forms. Regarding hot forms, we will make use of Theorem \ref{th:stops} to deal with a hypothetical Left answer to $G^{R_1}+H+*(2^n-k)$. Regarding tepid forms and infinitesimals we will make use of Theorem \ref{th:remoteness} and our new concept of distance (Section \ref{sec:dist}) to deal with a hypothetical Left answer to $G^{R_1}+H+*(2^n-k)$.

\begin{theorem}\label{th:nimbers}Let $*k$ be a nimber. Then $*k$ is indecomposable if and only if $k$ is a power of $2$.
\end{theorem}

\begin{proof}
($\Rightarrow$) Assume that $k$ is not a power of two. By letting $2^j$ be the largest power of two strictly smaller than $k$, we can write $k$ as $2^j + (k-2^j)$ and $*(2^j) + *(k-2^j)$ is a strong decomposition of $*k$.\\

\noindent
($\Leftarrow$) Suppose that $G+H=*(2^n)$ is a decomposition of $*(2^n)$, that $G$ and $H$ are canonical forms, and that $b(G)+b(H)$ is minimum. By Theorem \ref{th:nimsum}, neither $G$ nor $H$ can be a nimber. Therefore, $G$ and $H$ can only be numbers, hot forms, or tepid forms. We prove that each case is impossible.

If $G$ is a number, then, by Theorem \ref{th:brnumbers}, we have $b(H)=b(G)+2^n \geqslant b(\ast(2^n))$, which contradicts the assumption that $G+H=\ast(2^n)$ is a decomposition. Suppose $G$ is hot, meaning that $\LS{G}>\RS{G}$. Since we are assuming that $G+H=*(2^n)$, we have $\LS{*(2^n)}=\RS{*(2^n)}=0$, and $\LS{G+H}=\RS{G+H}=0$.  From Theorem \ref{th:stops}, we have that $\LS{G}+\RS{H}\leqslant \LS{G+H}=0$, and $0=\RS{G+H}\leqslant \LS{G}+\RS{H}$. Therefore, we can conclude that $\LS{G}\leqslant -\RS{H}\leqslant \LS{G}$, i.e., $\LS{G}=-\RS{H}$. We can deduce that $\RS{G}=-\LS{H}$ in a similar way. Since $G+H+*(2^n)=0$, i.e., a $\mathcal{P}$-position, if Right chooses to move to $G^{R_1}+H+*(2^n)$, where $\RS{G}=\LS{G^{R_1}}$, Left must have a winning reply. If Left moves to $G^{R_1}+H+*(2^n-j)$, then Right can answer with $G^{R_1}+H^{R_2}+*(2^n-j)$, where $\RS{H}=\LS{H^{R_2}}$.  However,
\begin{eqnarray*}
\LS{G^{R_1}+H^{R_2}}&\leqslant& \LS{G^{R_1}}+\LS{H^{R_2}},\hfill\text{ by Theorem \ref{th:stops}}\\
&=&\RS{G}+\RS{H}\\
&=&\RS{G}-\LS{G}.
\end{eqnarray*}
Since $G$ is hot and both $\RS{G}$ and $\LS{G}$ are dyadic rationals, then
\[\LS{G^{R_1}+H^{R_2}}\leqslant \RS{G}-\LS{G}=x<0,\]
 for some dyadic rational $x$. By using Theorem \ref{th:stops},
 we deduce the inequality $G^{R_1}+H^{R_2}+*(2^n-j)\leqslant x+*(2^n-j)<0$. Consequently, Left's move to $*(2^n-j)$ is not a winning move. Now, suppose a Left winning move is $G^{R_1L}+H+*(2^n)$, i.e., $G^{R_1L}+H+*(2^n)\geqslant 0.$  By combining this inequality with $G+H+*(2^n)=0$, we obtain $G^{R_1L}\geqslant G$. Consequently, $G^{R_1}$ is a reversible option, contradicting the assumption that $G$ is in canonical form. The only possibility for a winning move that remains is $G^{R_1}+H^L+*(2^n)>0$, where $\LS{H}=\RS{H^L}$. The inequality is strict because $b(G)+b(H)$ is minimum. Now consider a Right winning move in $G+H^L+*(2^n)$. The previous arguments, with $G$ and $H$ interchanged, show that Right's winning response must be in $G$ giving $G^{R_2}+H^L+*(2^n)< 0$. However, this inequality with $G^{R_1}+H^L+*(2^n)>0$ shows that $G^{R_1}$ is a dominated option in $G$, a contradiction.

 Since $G$ and, by symmetry,  $H$ cannot be numbers or hot, the only case left is both are tepid.  Thus, according to that, suppose that $G$ and $H$ are tepid. In short, let us suppose that $G=x+G'$ and $H=y+H'$ where $x$ and $y$ are numbers, and $G'$ and $H'$ are canonical forms of infinitesimals. By assumption, $\LS{G+H+*(2^n)}=\LS{0}=0$ and $\LS{G'+H'+*(2^n)}=0$ since the disjunctive sum of infinitesimals is infinitesimal. By Theorem~\ref{th:stops}, it follows that
$\LS{G+H+*(2^n)}=x+y+\LS{G'+H'+*(2^n)}$, i.e.,\linebreak $0=x+y+0$. Thus, $G=x+G'$, and $H=-x+H'$. Theorem~\ref{th:nimsum} guarantees that if either $G'$ or $H'$ is a nimber, then one of them, say $G'$, is $*k$ with $k\geqslant2^n$. According to Corollary~\ref{th:birthnimber}, $b(G) = |x|+k\geqslant2^n$, which ensures that $G+H$ is not a decomposition of $*(2^n)$, and leads to a contradiction. Hence, neither $G'$ nor $H'$ can be a nimber. Note also that, by Lemma~\ref{lem:remoteness}, $*(2^{n}-1)$ is remote for $G'$ and for $H'$. Now, without loss of generality, we may assume that $G'\cglfuz *(2^{n})$, since if $G'\cggfuz *(2^{n})$, the argument is analogous. Using Theorem~\ref{th:remoteness}, we can infer that, for all $j\geqslant 2^{n}-1$,  $G'\cglfuz *j$. Since $G'\cglfuz *(2^{n}-1)$, we know that $-G'+*(2^{n})\cggfuz *(2^{n+1}-1)$, and therefore, $H'=-G'+*(2^{n})\cggfuz *(2^{n+1}-1)$. By Theorem~\ref{th:remoteness} again, we conclude that, for all $j\geqslant 2^{n}-1$,  $H'\cggfuz *j$. Additionally, for any $k<2^n$, we have $G'\cglfuz*k+*2^n$, which implies $G'+*2^n\cglfuz*k$, and thus $-H'\cglfuz*k$ implies $H'\cggfuz*k$. Therefore, $H'\cggfuz \cgallstar$, and, with analogous reasoning, $G\cglfuz \cgallstar$. It follows that $\RD{H'}\geqslant 1$ and $\LD{G'}\geqslant 1$. Moreover, Corollary~\ref{th:twoahead2} ensures that $\RD{H'}=\LD{G'}\geqslant 1$. Now, let us suppose that in the position $G'+H'+*(2^n)$ Right moves to $G'+H'^{R_1}+*(2^n)$ with $\LS{H'^{R_1}}= 0$, decreasing the distance in the second component. If Left responds to some $G'+H'^{R_1}+*k$ where $k<2^n$, she cannot win. By Theorem \ref{th:twoahead}, in that position, Right can play and win. Left cannot reply with any $G'+H'^{R_1L}+*(2^n)\geqslant 0$ since this results in a reversible option in $G'$ when combined with $G'+H'+*(2^n)=0$. Left cannot reply with any $G'^L+H'^{R_1}+*(2^n)>0$ where $\RS{G'^L}<0$, while $\LS{H'^{R_1}}= 0$ and $\LS{*(2^n)}=0$,
then, by Theorem~\ref{th:stops}, we have $\RS{G'^L+H'^{R_1}+*(2^n)}\leqslant\LS{H'^{R_1}}+\RS{G'^L+*(2^n)}=
\RS{G'^L+*(2^n)}\leqslant \RS{G'^L}+\LS{*(2^n)}<0$. This means that Right wins. Therefore, Left must try to find a move $G'^L$ such
that $\RS{G'^L}=0$. In this case, $\LD{G'^L}$ is well defined, and due to Theorem \ref{th:twoahead}, Left must be able to reduce
the distance. However, Left cannot win by answering any $G'^L+H'^{R_1}+*(2^n)>0$ that reduces the distance in the first component. Observe that the inequality is strict since, by Theorem~\ref{th:translationrev}, $x+G'^L\in G^{\mathcal{L}}$, $-x+H'^{R_1}\in H^{\mathcal{R}}$, and $b(G)+b(H)$ is minimum. If that were possible, using similar arguments, Right would need to find a winning move $G'^L+H'^{R_2}+*(2^n)<0$ against Left's first move in $G'+H'+*(2^n)$ to $G'^L+H'+*(2^n)$ . Joining this inequality with $G'^L+H'^{R_1}+*(2^n)>0$ would yield a dominated option in $H'$. Since Left has no way of emerging victorious against Right's move to $G'+H'^R+*(2^n)$, we arrive at a contradiction.

All cases, i.e., $G$ and $H$ being numbers, hot forms, nimbers or tepid forms,  gave contradictions, hence, the theorem is proved.
\end{proof}

\begin{theorem}\label{th:strongnimbers}Let $k$ be a nonnegative integer. If $k$ is not a power of two then $*k$ is strongly decomposable and the only strong decompositions of $*k$ are sums of nimbers.~
\end{theorem}

\begin{proof}
All strong decompositions $*k=G+H$ are minimal in terms of \linebreak $b(G)+b(H)$. Therefore, if $G$ and $H$ are not nimbers, then all contradictions found in the proof of Theorem \ref{th:numbers} show that a strong decomposition $*k=G+H$ cannot exist. If both $G$ and $H$ are nimbers, we can let $2^j$ be the largest power of two strictly smaller than $k$, and consider $G=*(2^j)$ and $H=*(k-2^j)$. By Theorem \ref{th:nimsum}, we have $b(G)=2^j$, $b(H)=k-2^j$, $b(*k)=k$, $b(G)+b(H)=b(*k)$, and $*k=G+H$ is a strong decomposition.
\end{proof}

\section{Final remarks}
\label{sec:final}

Decompositions of numbers and nimbers are already used implicitly, when analyzing games.

In a disjunctive sum of numbers  represented by mixed fractions \linebreak$m_1\bm{\frac{r_1}{2^{n_1}}}+\ldots+m_k\bm{\frac{r_k}{2^{n_k}}}$, it is easy to determine the outcome. If the sum is positive, Left wins, if negative, then Right wins, and if it is zero, then $G$ is a second player win. The standard texts \cite{ANW019,BCG001,Con001,Sie013} claim that the game is over, however, even if the components are in canonical form, there is still a decision to be made in this type of endgame. We make this explicit. If all the numbers are integers then playing in any has the same effect--the sum changes by $-1$ if Left plays, and $+1$ if Right. Otherwise, move in the fraction that has the greatest denominator. Consider a dyadic $m+\frac{r}{2^n}$, where $m$ is an integer, $0<r=2j+1<2^n$, and $j\geqslant 0$. Since the canonical form of that component is $\{m\frac{j}{2^{n-1}}\,|\,m\frac{j+1}{2^{n-1}}\}$, the change in the sum will be $-\frac{1}{2^n}$ if Left plays on it and $\frac{1}{2^n}$ if Right plays on it. For example,
let
\begin{eqnarray*}
G&=&1\bm{\frac{7}{8}}-1\bm{\frac{1}{2}}-\bm{\frac{1}{4}}=\frac{1}{8}\\
   &=& \left(1+\frac{7}{8}\right)+\left(-2+\frac{1}{2}\right)+\left(-1+\frac{3}{4}\right).\end{eqnarray*}
 Clearly, $G=\frac{1}{8}>0$ and Left wins. However, Left cannot play in any component that decreases the sum by more than $\frac{1}{8}$. That is, Left must play on the first component, reducing the sum to zero, which Left wins playing second. Playing on the second component decreases the sum to
$-\frac{3}{8}$, and playing on the third component decreases the sum to $-\frac{1}{8}$, both of which Left loses.\footnote{In fact, it is possible to prove a ``Greatest Denominator Choice Theorem'' even for sums of numbers that are not in canonical form.}
It is interesting to observe that the indecomposable numbers (absolute value not exceeding $1$) are crucial to find the good moves. The algebraic reason for this lies in the fact that the incentive of a non-integer component $m\bm{\frac{r}{2^{n}}}$ is the incentive of its indecomposable fractional part $\bm{\frac{r}{2^{n}}}$. In the case where all components are non-zero integers, all components have an incentive equal to the incentive of the only non-zero indecomposable integers that exist, which are $1$ and $-1$.

In a disjunctive sum of nimbers $*m_1+\ldots+*m_k$, which represents an impartial position, either all components are equal to $*$ or the Grundy value of at least one component is greater than $1$. If all components are equal to $*$, all moves are equally good. If this is not the case, the proper procedure is to decompose the components according to the binary representations of their sizes and then cancel the powers of $2$ in pairs. For example, the sum $*7+*5+*9$ is equal to $(*4+*2+*)+(*4+*)+(*8+*)$, which, in turn, is equal to $(\cancel{*4}+*2+\cancel{*})+(\cancel{*4}+\cancel{*})+(*8+*)=*11$. The first player can win by reducing the size of the last nimber to $2$. It is interesting to observe that the indecomposable nimbers (sizes equal to powers of $2$) are crucial to find the good moves. The algebraic reason for this lies in the fact that powers of $2$ sustain binary representations and cancellations in pairs are essentially the definition of the \textsc{nim} sum, which has been proven to be the determining operation for these cases.

Now that Theorems \ref{th:numbers}, \ref{th:strongnumbers}, \ref{th:nimbers}, and \ref{th:strongnimbers} have been proved, it is natural to think of games of the type $x+*n$, where $x$ is a number and $*n$ is a nimber. A consequence of Corollary \ref{th:birthnimber} is that these games are strongly decomposable. This corollary can be given an interesting interpretation if we turn our attention to \textsc{blue-red-green hackenbush} strings. It is not particularly difficult to prove that if $H$ is a \textsc{blue-red-green hackenbush} string with $n$ edges, then $b(H)=n$. As a consequence of this, if $G$ and $H$ are two \textsc{blue-red-green hackenbush} strings such that $G \os H=G+H$, then $G+H$ is a strong decomposition of $G \os H$.\footnote{The symbol ``$\os$'' designates the \emph{ordinal sum}; if a player moves on the bottom, the top disappears, while if a player moves on the top, nothing happens to the bottom \citep{ANW019,BCG001,Con001,Sie013}.} That happens because $b(G \os H)=b(G)+b(H)$ simply reflects the fact that the number of edges of $G \os H$ is the sum of the numbers of edges of $G$ and $H$. By using Lemma 4.3.4 of \cite{Mc016}, it turns out that $x \os \cgstar n = x + *n$, and, thus, $b(x+*n) = b(x)+b(*n) = b(x)+n$, which is precisely the statement of Corollary \ref{th:birthnimber}. It is also worth noting that \textsc{blue-red-green hackenbush} strings could have also been used to prove both Theorem~\ref{th:brnumbers} and the first item of Theorem \ref{th:nimsum}. Moreover, the indecomposable numbers correspond to \textsc{blue-red hackenbush} strings with two bottommost edges of different colors, while the decomposable numbers correspond to \textsc{blue-red hackenbush} strings with two bottommost edges of the same color. Although Theorems \ref{th:numbers} and \ref{th:nimbers}, the main contributions of this paper, concern decompositions of any kind (both strong and not strong), these considerations point to future work that can be done on the use of ordinal sums and rulesets like \textsc{blue-red-green hackenbush} to explore whether certain games are strongly decomposable.

Just like a primality test, it seems overly ambitious to seek an expeditious test for assessing the decomposability of an arbitrary short game. This idea can already be supported by the preceding paragraphs. The algebraic reasons for the importance of indecomposable numbers and indecomposable nimbers in disjunctive sums are considerably distinct. This suggests that the fundamental nature of an indecomposable game may not be general, but rather dependent on a more restricted class to which that indecomposable game belongs. Nevertheless, there are classes of games widely studied in specialized literature, such as switches, tinies and minies, uptimals, and so on. As seen here, indecomposable games seem to bring something essential with them. In the future, it would be interesting to expand the work to other classes, with a view to better understanding the algebra of the group of short combinatorial games.

\end{document}